\documentclass[a4paper,12pt]{amsart}

\usepackage{amsmath}
\usepackage{amsthm}
\usepackage{amssymb}
\usepackage{amsfonts}
\usepackage{mathrsfs}  
\usepackage{enumitem} 
\usepackage{color}
\usepackage[bookmarks=true,hyperindex,pdftex,colorlinks,citecolor=blue]{hyperref}
\usepackage{marginnote} 
\usepackage{graphicx} 
\usepackage{caption} 
\usepackage{soul} 

\DeclareMathOperator{\lspan}{span}       
\DeclareMathOperator{\conv}{conv}        
\DeclareMathOperator{\Lip}{Lip}          
\DeclareMathOperator{\lip}{lip}          

\DeclareMathOperator{\Br}{Br}            

\newcommand{\NN}{\mathbb{N}}             
\newcommand{\ZZ}{\mathbb{Z}}             
\newcommand{\QQ}{\mathbb{Q}}             
\newcommand{\RR}{\mathbb{R}}             

\newcommand{\abs}[1]{\left|{#1}\right|}                     
\newcommand{\pare}[1]{\left({#1}\right)}                    
\newcommand{\set}[1]{\left\{{#1}\right\}}                   
\newcommand{\norm}[1]{\left\|{#1}\right\|}                  
\newcommand{\dual}[1]{{#1}^\ast}                            
\newcommand{\ball}[1]{B_{{#1}}}                             
\newcommand{\duality}[1]{\left<{#1}\right>}                 
\newcommand{\cl}[1]{\overline{#1}}                          
\newcommand{\wconv}{\stackrel{w}{\rightarrow}}              

\newcommand{\F}[1]{\mathcal{F}({#1})}        
\newcommand{\Free}{\mathcal F}
\newcommand{\lipfree}[1]{\mathcal{F}({#1})}  
\newcommand{\lipnorm}[1]{\norm{#1}_L}        

\theoremstyle{plain}
\newtheorem{theorem}{Theorem}[section]
\newtheorem{lemma}[theorem]{Lemma}
\newtheorem{corollary}[theorem]{Corollary}
\newtheorem{proposition}[theorem]{Proposition}
\newtheorem{claim}{Claim}
\newtheorem{fact}[theorem]{Fact}

\newtheorem*{T1}{Theorem~\ref{tm:separable_tree_embedding}}
\newtheorem*{T2}{Theorem~\ref{t:char_of_embeddability}}
\newtheorem*{T3}{Theorem~\ref{th:lipfree_tree_embedding}}

\theoremstyle{definition}
\newtheorem*{definition*}{Definition}
\newtheorem{definition}[theorem]{Definition}
\newtheorem{example}[theorem]{Example}
\newtheorem{question}{Question}

\theoremstyle{remark}
\newtheorem{remark}[theorem]{Remark}

\begin{document}

\title{Embeddings of Lipschitz-free spaces into $\ell_1$}

\author[R. J. Aliaga]{Ram\'on J. Aliaga}
\address[R. J. Aliaga]{Universitat Polit\`ecnica de Val\`encia, Instituto Universitario de Matem\'atica Pura y Aplicada, Camino de Vera S/N, 46022 Valencia, Spain}
\email{raalva@upvnet.upv.es}

\author[C. Petitjean]{Colin Petitjean}
\address[C. Petitjean]{LAMA, Univ Gustave Eiffel, UPEM, Univ Paris Est Creteil, CNRS, F-77447, Marne-la-Vall\'ee, France}
\email{colin.petitjean@u-pem.fr}

\author[A. Proch\'azka]{Anton\'in Proch\'azka}
\address[A. Proch\'azka]{Laboratoire de Math\'ematiques de Besan\c con,
Universit\'e Bourgogne Franche-Comt\'e,
CNRS UMR-6623,
16, route de Gray,
25030 Besan\c con Cedex, France}
\email{antonin.prochazka@univ-fcomte.fr}

\date{} 


\begin{abstract}
We show that, for a separable and complete metric space $M$, the Lipschitz-free space $\F M$ embeds linearly and almost-isometrically into $\ell_1$ if and only if $M$ is a subset of an $\RR$-tree with length measure 0. Moreover, it embeds isometrically if and only if the length measure of the closure of the set of branching points of $M$ (taken in any minimal $\RR$-tree that contains $M$) is negligible. 
We also prove that, for any subset $M$ of an $\RR$-tree, every extreme point of the unit ball of $\F M$ is an element of the form $(\delta(x)-\delta(y))/d(x,y)$ for $x\neq y\in M$. 
\end{abstract}


\subjclass[2010]{Primary 46B20, 05C05; Secondary 46B25, 54C25}



\keywords{embedding, extreme point, length measure, Lipschitz-free space, Lipschitz homeomorphism, $\RR$-tree}

\maketitle


\section{Introduction}

Our goal in this paper is to provide some contributions to the isometric and isomorphic classification of Lipschitz-free Banach spaces. Let us start by giving some necessary definitions. Given a \emph{pointed metric space} $(M,d)$, i.e. one where we have selected an element $0$ as a base point, we consider the space $\Lip_0(M)$ of all real-valued Lipschitz functions on $M$ that map the base point to $0$. $\Lip_0(M)$ is then a Banach space endowed with the norm given by the Lipschitz constant
$$
\lipnorm{f}=\sup\set{\frac{\abs{f(x)-f(y)}}{d(x,y)}:x,y\in M,x\neq y}
$$
for $f\in\Lip_0(M)$. Let $\delta\colon M\rightarrow\dual{\Lip_0(M)}$ map each $x\in M$ to its evaluation functional given by $\duality{f,\delta(x)}=f(x)$ for $f\in\Lip_0(M)$. Then $\delta$ is an isometric embedding of $M$ into $\dual{\Lip_0(M)}$, and the space $\lipfree{M}=\cl{\lspan}\,\delta(M)$ is called the \emph{Lipschitz-free space} over $M$. It is well known that $\lipfree{M}$ is an isometric predual of $\Lip_0(M)$; see the monograph \cite{Weaver2} for reference (where $\lipfree{M}$ is denoted $\text{\AE}(M)$).

The structure of Lipschitz-free spaces is not completely understood to this day. One way to advance this knowledge is to study the possible embeddings of classical Banach spaces into Lipschitz-free spaces or vice versa. A major step in this direction was given by Godard \cite{Godard_2010} when he proved that $\lipfree{M}$ can be linearly isometrically embedded into an $L_1$ space if and only if $M$ can be isometrically embedded into an $\RR$-tree. Let us just say for now that an $\RR$-tree is a metric space where each pair of points is connected by a unique path that is isometric to a segment of~$\RR$. 

In order to discuss the extensions of Godard's result, we briefly introduce two necessary concepts related to $\RR$-trees that will be developed further later. First, a canonical measure $\lambda$ called the \emph{length measure} may be defined on such a space that extends the concept of Lebesgue measure in $\RR$. Second, an element $x$ of an $\RR$-tree $T$ is said to be a \emph{branching point} of $T$ if $T\setminus\set{x}$ has at least three connected components.

Let us now consider the related problem of characterizing all complete metric spaces $M$ such that $\lipfree{M}$ can be isometrically embedded into $\ell_1$. Since $\ell_1$ is contained in $L_1$, such a metric space must be a subset of an $\RR$-tree. Godard also showed that if $A$ is a subset of $\RR$ with positive measure then $\lipfree{A}$ contains an isometric copy of $L_1$. It follows easily that the same is true for $M$ if $\lambda(M)>0$. Hence $M$ must also be \emph{negligible}, i.e. $\lambda(M)=0$. The question immediately arises whether these necessary conditions are sufficient:

\begin{question}
\label{q:embedding_into_l1}
If $M$ is a closed 
subset of an $\RR$-tree such that $\lambda(M)=0$, is it true that $\lipfree{M}$ embeds isometrically into some $\ell_1(\Gamma)$?
\end{question}

In \cite{DaKaPr_2016}, Dalet, Kaufmann and Proch\'azka provided further insight into this problem. They showed that $\lipfree{M}$ is isometric to $\ell_1(\Gamma)$ if and only if $M$ is negligible and moreover contains all branching points. As a consequence, they gave a positive answer to Question \ref{q:embedding_into_l1} for compact $M$ based on the following observation: if $M$ is compact and negligible, then the closure of the set of branching points 
(taken in any minimal $\RR$-tree that contains $M$) 
is also negligible.
However, easy examples show that the observation is not valid in general for non-compact $M$.

Our first main result 
shows that the failure of this observation indeed creates an obstacle for isometric embedding of $\Free(M)$ into $\ell_1$. In what follows, $\Br(M)$ denotes the set of branching points of $M$.

\begin{theorem}\label{t:char_of_embeddability}
Let $M$ be a complete metric space. Then the following are equivalent:
\begin{enumerate}[leftmargin=1cm, label={\upshape{(\roman*)}}]
\item $\lipfree{M}$ is isometrically isomorphic to a subspace of 
$\ell_1(\Gamma)$ for some set $\Gamma$,
\item $M$ is a subset of an $\RR$-tree such that $\lambda(M)=0$ and $\lambda(\cl{\Br(M)})=0$.
\end{enumerate}
\end{theorem}

Despite the negative answer to Question~\ref{q:embedding_into_l1}, the second
main result in this paper states that, when $M$ is separable, $\lipfree{M}$ embeds \emph{almost isometrically} into $\ell_1$ under the specified hypotheses:

\begin{theorem}
\label{th:lipfree_tree_embedding}
Let $M$ be a complete separable metric space. Then the following are equivalent:
\begin{enumerate}[leftmargin=1cm, label={\upshape{(\roman*)}}]
\item $\lipfree{M}$ is $(1+\varepsilon)$-isomorphic to a subspace of
$\ell_1$ for every $\varepsilon>0$,
\item $M$ is a subset of an $\RR$-tree such that $\lambda(M)=0$.
\end{enumerate}
\end{theorem}

The proof of the 
implication (ii)$\Rightarrow$(i)
in Theorem~\ref{th:lipfree_tree_embedding} is constructive. It consists in perturbing $M$ iteratively with arbitrarily small distortion in such a way that the end result has its branching points confined into a closed negligible set. The proof of the other direction is based on a standard ultraproduct argument and Godard's theorem.

As a consequence, we obtain the equivalence of the Schur and Radon-Nikod\'ym properties for Lipschitz-free spaces over subsets of $\RR$-trees. Let us remark that they are known not to be equivalent in general for subspaces of $L_1$
(see e.g. \cite{BoRo_1980}), and that it is an open problem whether they are equivalent for all Lipschitz-free spaces.

In relation to this, we also prove that for complete, proper subsets of $\RR$-trees (i.e. such that all closed balls in $M$ are compact), being negligible is equivalent to $\lipfree{M}$ being a dual Banach space; see Theorem \ref{th:proper_trees}.

Finally, we extend a result of Kadets and Fonf \cite{KaFo_1983} about subspaces of $\ell_1$ to prove that all extreme points of the unit ball of a subspace of an $L_1$ space are always preserved, i.e. they are also extreme points of the bidual ball. Combining this with previous results by Aliaga and Perneck\'a \cite{AlPe_2019}, we completely characterize the extreme points of the unit ball of $\lipfree{M}$ for any subset $M$ of an $\RR$-tree (cf. Corollary \ref{cr:extreme_points_trees}).

Let us conclude this exposition with the obvious remark that $\RR$ is itself an $\RR$-tree, so all results in this paper apply in particular to closed subsets $M$ of $\RR$.


\subsection{Preliminaries}

Let us summarize the necessary background and notation used in this paper. Given a Banach space $X$, $B_X$ will be its closed unit ball. Let us remark that we will consider exclusively real scalars. We let $(M,d)$ be a metric space and if $M$ is pointed, its base point will be called $0$. We will denote by
\begin{align*}
d(x,A) &= \inf\set{d(x,a):a\in A} \\
d(A,B) &= \inf\set{d(a,b):a\in A,b\in B}
\end{align*}
the distance between a subset and either an element or another subset of $M$. The characteristic function of a set $A$ will be denoted by $\mathbf 1_A$.

We now introduce $\RR$-trees properly. An \emph{$\RR$-tree} is an arc-connected metric space $(T,d)$ with the property that there is a unique arc connecting any pair of points $x\neq y\in T$ and it moreover is isometric to the real segment $[0,d(x,y)]\subset\RR$. Such an arc, denoted $[x,y]$, is called a \emph{segment} of $T$ and it is immediate that it coincides with the metric segment
$$
[x,y]=\set{p\in T: d(x,p)+d(p,y)=d(x,y)} .
$$
Given a segment $I=[x,y]$, we will write $I^o=(x,y)$ for its interior.

A point $x\in T$ is called a \emph{leaf} of $T$ if $T\setminus\set{x}$ is connected, and it is called a \emph{branching point} of $T$ if $T\setminus\set{x}$ has at least three connected components. The set of all branching points of $T$ is denoted $\Br(T)$. If $T$ is a separable $\RR$-tree, then $\Br(T)$ is at most countable, and for each $b\in\Br(T)$ the set $T\setminus\set{b}$ has at most countably many connected components; see \cite{MaNiOv_1992} for reference.

The isometries $\phi_{xy}\colon [x,y]\rightarrow [0,d(x,y)]$ allow us to define an analog of the Lebesgue measure, called the \emph{length measure}, on $T$ as follows. Given an interval $[x,y]$ in $T$, let us say that a set $E\subset [x,y]$ is measurable if $\phi_{xy}(E)$ is Lebesgue measurable. Next define its length measure as $\lambda(E)=\lambda(\phi_{xy}(E))$, where $\lambda$ also denotes the Lebesgue measure on $\RR$ (which coincides with its length measure when considering $\RR$ as an $\RR$-tree). For an arbitrary $E\subset T$, let us say that $E$ is measurable if $E\cap I$ is measurable for any segment $I$ in $T$, and define its length measure as
$$
\lambda(E)=\sup\set{\sum_{k=1}^n \lambda(E\cap I_k): \text{$I_k$ are disjoint segments in $T$}} .
$$

It is well known that $\RR$-trees satisfy the following \emph{four point condition}: for any $x,y,z,w\in T$, the inequality
\begin{equation}
\label{eq:4pc}
d(x,y)+d(z,w)\leq\max\set{d(x,z)+d(y,w),d(y,z)+d(x,w)}
\end{equation}
holds. It is immediate that any subset of an $\RR$-tree also satisfies this condition. Conversely, any metric space $M$ satisfying the four point condition can be realized as a subspace of an $\RR$-tree \cite{Buneman_1974}. In fact, up to isometry there is a unique minimal $\RR$-tree containing $M$; we will denote it by $\conv(M)$. If $M$ is a subset of an $\RR$-tree $T$, then $\conv(M)$ may be realized as the union of the segments $[x,y]\subset T$ for $x,y\in M$, or alternatively as the union of the segments $[p,x]$, $x\in M$ for any fixed $p\in M$; in particular, if $M$ is separable then so is $\conv(M)$. We may then uniquely define the length measure on $M$ as the restriction to $M$ of the length measure on $\conv(M)$. We will also denote $\Br(M)=\Br(\conv(M))$; note that branching points of $M$ do not necessarily belong to $M$.

We will consider any $\RR$-tree $T$ to have a designated base point $0$ which we shall call its \emph{root}. This allows us to define a partial order $\preccurlyeq$ on $T$ by saying that $x\preccurlyeq y$ if $x\in [0,y]$. We will also use the notation $x\prec y$ to say that $x\preccurlyeq y$ and $x\neq y$. Given any two points $x,y\in T$, their order-theoretic meet $x\wedge y$ exists in $T$ with the property that, for any $z\in T$, $z\preccurlyeq x\wedge y$ if and only if $z\preccurlyeq x$ and $z\preccurlyeq y$; it is given by $[0,x\wedge y]=[0,x]\cap [0,y]$. If neither $x\prec y$ nor $y\prec x$ is true, then $x\wedge y\in\Br(T)$. Let us also mention that for any $p\in T$, the map $x\mapsto x\wedge p$ is continuous. Note that $\preccurlyeq$ induces (by restriction) a corresponding partial order on any subset $M\subset T$, with the difference that meets do not necessarily belong to $M$.

When considering subsets of $T$ (that contain 0) as pointed metric spaces, the base point will always be assumed to be the root. We will say that a mapping $\psi\colon M\rightarrow N$ between subsets of $\RR$-trees \emph{preserves the order} if $p\preccurlyeq q$ implies $\psi(p)\preccurlyeq\psi(q)$ for any $p,q\in M$. We will say that it is an \emph{order isomorphism} if moreover $\psi(p)\preccurlyeq\psi(q)$ also implies that $p\preccurlyeq q$. Notice that any root-preserving isometry between subsets of $\RR$-trees is an order isomorphism, since the order relation is completely determined by the metric and the choice of the root. In particular, if $x\prec y$ then $\phi_{xy}$ is an order isomorphism, and so it makes sense to consider infima and suprema of subsets of the segment $[x,y]\subset T$.

Finally, let us collect a few additional facts about $\RR$-trees and their subsets that will be useful later. We omit the simple proofs.

\begin{fact}
\label{fact:dense_br}
If $M$ is a subset of an $\RR$-tree and $A$ is a dense subset of $M$ then $\Br(A)=\Br(M)$.
\end{fact}

\begin{fact}
\label{fact:order_preservation}
If $M,N$ are subsets of $\RR$-trees and $\psi\colon M\rightarrow N$ preserves the order, then $\psi(\conv(A))\subset\conv(\psi(A))$ for any $A\subset M$.
\end{fact}

\begin{fact}
\label{fact:metricr_projection}
Let $T$ be a complete $\RR$-tree that is a subspace of an $\RR$-tree~$T'$.
Then there is a unique metric projection $\pi_T\colon T' \to T$. 
In particular if $I$ is a segment in an $\RR$-tree $T'$, then there is a metric projection (or 1-retraction) $\pi_I\colon T'\rightarrow I$.
\end{fact}

\begin{fact}
\label{fact:union_completion}
The nested union of $\RR$-trees, the intersection of $\RR$-trees and the completion of an $\RR$-tree are again $\RR$-trees.
\end{fact}

\begin{fact}\label{fact:union}
Let $(T_i,d_i)$, $i\in I$, be complete $\RR$-trees such that $\bigcap_{i\in I}T_i\neq \varnothing$ and $d_i=d_j$ on $T_i \cap T_j$ for all $i,j \in I$.
We define a metric $d$ on $T'=\bigcup T_i$ by
$$
d(x,y)=d_i(x,\pi_{ij}(x))+d_i(\pi_{ij}(x),\pi_{ij}(y))+d_j(\pi_{ij}(y),y)
$$
where $i,j\in I$ are such that $x\in T_i$ and $y \in T_j$, and $\pi_{ij}=\pi_{T_i\cap T_j}$.
Then $d$ is well defined, $(T',d)$ is an $\RR$-tree and each $(T_i,d_i)$ is a metric subspace of $(T',d)$. 
\end{fact}


\section{Lipschitz free spaces over negligible subsets of \texorpdfstring{$\RR$}{R}-trees}

\subsection{Isometric embeddings into \texorpdfstring{$\ell_1(\Gamma)$}{l1(Gamma)}}


Let us start with the full solution to Question~\ref{q:embedding_into_l1}. For the convenience of the reader we restate the result here. 
\begin{T2}
Let $M$ be a complete metric space. Then the following are equivalent:
\begin{enumerate}[leftmargin=1cm, label={\upshape{(\roman*)}}]
\item $\lipfree{M}$ is isometrically isomorphic to a subspace of 
$\ell_1(\Gamma)$ for some set $\Gamma$,
\item $M$ is a subset of an $\RR$-tree such that $\lambda(M)=0$ and $\lambda(\cl{\Br(M)})=0$.
\end{enumerate}
\end{T2}

For a compact metric space $M$ which is a subset of an $\RR$-tree, $\lambda(M)=0$ already implies $\lambda(\cl{\Br(M)})=0$ by~\cite[Lemma 7]{DaKaPr_2016}.
Also, there are easy examples of proper 
$M$ such that $\lambda(M)=0$ and $\lambda(\cl{\Br(M)})>0$ (see Example~\ref{ex:dense_branching} below).
Finally, let us mention that general results concerning subspaces of $L_1$ which are isometric to subspaces of $\ell_1$ appear in~\cite{DeJaPe_1998}.
We prefer to provide a direct self-contained proof here.

\begin{proof}[Proof of Theorem \ref{t:char_of_embeddability}]
The implication
(ii)$\Rightarrow$(i)
follows immediately from Godard's work \cite[Theorem 3.2]{Godard_2010}. We now turn to
(i)$\Rightarrow$(ii).
If $\lipfree{M}$ embeds isometrically into $\ell_1(\Gamma)$ for some set $\Gamma$, then it follows again from Godard's work that $M$ is a subset of an $\RR$-tree such that $\lambda(M)=0$. 
So we will assume that $\lambda(\cl{\Br(M)})>0$ and then show that $\Free(M)$ does not embed linearly isometrically into any $\ell_1(\Gamma)$. To achieve this goal, we are going to present a two-dimensional subspace of $\mathcal F(M)$ that does not embed linearly isometrically into $\ell_1$ (and so not into any $\ell_1(\Gamma)$ either).
First, observe a particularity of 2-dimensional subspaces of $\ell_1$:

\begin{lemma}
Let $u,v \in \ell_1\setminus \set{0}$. Then the second distributional derivative of  $\mathbb R\ni t \mapsto \norm{u-tv}_1$ is a discrete measure (concentrated on the countable set $\set{\frac{u_i}{v_i}:i \in \NN, v_i\neq 0}$).
\end{lemma}
\noindent
The standard details of the easy proof are left to the reader. 

Now, since we assume that $\lambda(\cl{\Br(M)})>0$, 
there exists a segment $[a,b] \subset \conv(M)$ such that $\lambda([a,b]\cap \cl{\Br(M)}
)>0$.
Without loss of generality, we may assume that $0\preccurlyeq a\prec b$.

Let $\pi:\conv(M) \to [a,b]$ be the metric projection onto $[a,b]$.
Note that $\pi(\Br(M))\cap (a,b) \subset \Br(M)$.
Thus we can find a sequence $(q_n)_{n=1}^\infty \subset \Br(M)\cap (a,b)$ dense in $\cl{\Br(M)}\cap (a,b)$.

Since $b\in \conv(M)$, there exists $x\in M$ such that $b\preccurlyeq x$. 
We set $v=\delta(x)$.
Further, for each $n\in\NN$ there exists $x_n\in M$ such that $\set{q_n}=\Br(\set{0,x_n,x})$.
We set \[
u=\sum_{n=1}^\infty \frac{\delta(x_n)}{2^n d(0,x_n)}.
\]

To finish the proof, it is enough to show that the second distributional derivative of $t \mapsto \norm{u-tv}$ has a continuous part.
It follows from Godard's embedding that
\[
\norm{u-tv}=\int_0^{d(b,0)}\abs{f(s)-t
}\,ds + \sum_{n=1}^\infty \frac{d(q_n,x_n)}{2^n d(0,x_n)}+\abs{t}d(x,b)
\]
for every $t\in\RR$, where
$$
f=\sum_{n=1}^\infty \frac{1}{2^nd(0,x_n)}\mathbf 1_{[0,d(0,q_n)]} \in L_1[0,d(0,b)].
$$
For every test function $\varphi \in \mathcal D(\RR)$ we have 
\[
\begin{aligned}
\int_{\mathbb R}\pare{\int_0^{d(0,b)}|f(s)-t|\,d s} \varphi''(t) \,dt &= \int_0^{d(0,b)}\pare{\int_{\mathbb R} |f(s)-t|\varphi''(t)\,dt}ds\\
&=\int_0^{d(0,b)} 2\varphi(f(s))\,ds=\int_{\mathbb R}2\varphi(s)\,d\nu_f(s)
\end{aligned}
\]
where $\nu_f(B)=\lambda(f^{-1}(B))$ for all Borel subsets $B$ of $\mathbb R$, i.e. $\nu_f$ is the push-forward by $f$ of the Lebesgue measure.
Thus the second distributional derivative of $t \mapsto \norm{u-tv}$ is given by $2\nu_f$.
The set $[0,d(0,b)]\setminus \cl{\set{d(0,q_n):n \in \NN}}$ is a countable union of relatively open intervals $(I_n)_{n=1}^\infty$.
It is clear that for $A=[0,d(0,b)] \setminus \bigcup_{n=1}^\infty \cl{I_n}$ we have $\lambda(A)>0$ and it follows from the definition of $f$ and $A$ that $f^{-1}(\set{f(a)})=\set{a}$ for every $a\in A$. 
Thus $\nu_f$ is not discrete. 
\end{proof}

The following example satisfies the negation of the equivalent conditions stated in Theorem~\ref{t:char_of_embeddability}.

\begin{example}
\label{ex:dense_branching}
Let $(q_n)_{n=1}^\infty$ be an enumeration of $\QQ\cap (0,1)$, and let $T$ be an $\RR$-tree consisting of a segment $S$ of length 1, with the root at one of its ends, and a sequence of segments $(B_n)_{n=1}^\infty$ such that $B_n$ has length $n$ and is attached to $S$ at a distance $q_n$ from the root. Now let $M$ consist of the leaves of $T$. Clearly $M$ is closed, proper (its bounded subsets are finite), and countable, hence $\lambda(M)=0$. However $\Br(M)$ is dense in $S$ and so $\lambda(\cl{\Br(M)})=1$.

It thus follows from Theorem~\ref{t:char_of_embeddability} that $\F M$ does not embed isometrically into $\ell_1$. Nevertheless, it embeds almost isometrically into $\ell_1$. This follows from our Theorem~\ref{th:lipfree_tree_embedding} but we choose here to provide a very short proof based on the properness of $M$. By \cite[Proposition 4.3]{Kalton_2004}, the space $\lipfree{M}$ is linearly $(1+\varepsilon)$-isomorphic to a subspace of the $\ell_1$-sum of the spaces $\lipfree{M_k}$, $k\in\ZZ$, where $M_k=\set{x\in M: d(0,x)\leq 2^k}$. But $M_k$ are negligible compact sets, hence each $\lipfree{M_k}$ is isometric to a subspace of $\ell_1$ by the results in \cite{DaKaPr_2016} and this is enough.

If we let all segments $B_n$ have a constant length instead, a similar example is obtained where $M$ is bounded but not proper.
\end{example}

\subsection{Almost-isometric embeddings into \texorpdfstring{$\ell_1$}{l1}}

We now turn to the second main result. 
The proof will follow from two more general results where $M$ is not necessarily negligible. The first one is a version of the inverse direction of the equivalence stated in Theorem \ref{th:lipfree_tree_embedding}:

\begin{proposition}\label{p:ultraproduct}
Let $M$ be a pointed metric space such that for every $\varepsilon>0$ there is a measure space $(\Omega,\Sigma,\mu)$ such that $\Free(M)$ is $(1+\varepsilon)$-isomorphic to a subspace of $L_1(\mu)$.
Then $M$ is a subset of an $\RR$-tree.
\end{proposition}
\begin{proof}
For every $n \in \NN$, let $\mu_n$ be a measure such that there is a linear operator $T_n:\Free(M) \to L_1(\mu_n)$ satisfying
\[
\norm{u}\leq \norm{T_n(u)}_{L_1} \leq \left(1+\frac1n\right)\norm{u}
\]
for every $u\in \Free(M)$.
Let $\mathcal U$ be a non-principal ultrafilter on $\NN$ and let $(L_1(\mu_n))_{\mathcal U}$ be the corresponding ultraproduct of the spaces $L_1(\mu_n)$.
It is a well known fact that any ultraproduct of $L_1$-spaces is an $L_1$-space. (see~\cite[Theorem 3.3]{Heinrich_1980}).
We define $T(x)=[(T_n(x))_n] \in (L_1(\mu_n))_{\mathcal U}$.
It is standard to check that $T$ is a linear isometric embedding.
Thus, applying~\cite[Theorem 4.2]{Godard_2010} we get the conclusion.
\end{proof}

For the opposite direction, we will actually prove a more general result which shows that separable subsets $M$ of $\RR$-trees may be distorted with an arbitrarily small Lipschitz constant in order to concentrate the closure of its branching points around $M$ up to a negligible set. The precise statement follows:

\begin{theorem}
\label{tm:separable_tree_embedding}
Let $M$ be a complete separable metric space that is a subset of an $\RR$-tree. Then, for every $\varepsilon>0$, $M$ is $(1+\varepsilon)$-Lipschitz homeomorphic to a subset $N$ of an $\RR$-tree such that $\lambda(N)=\lambda(M)$ and $\lambda(\cl{\Br(N)}\setminus N)=0$, where the closure is taken in $\conv(N)$.
\end{theorem}

The proof of Theorem \ref{tm:separable_tree_embedding} will be provided in the next section. It is straightforward to obtain Theorem \ref{th:lipfree_tree_embedding} as a consequence. For convenience of the reader, we recall here the statement:

\begin{T3}
Let $M$ be a complete separable metric space. Then the following are equivalent:
\begin{enumerate}[leftmargin=1cm, label={\upshape{(\roman*)}}]
\item $\lipfree{M}$ is $(1+\varepsilon)$-isomorphic to a subspace of
$\ell_1$ for every $\varepsilon>0$,
\item $M$ is a subset of an $\RR$-tree such that $\lambda(M)=0$.
\end{enumerate}
\end{T3}

\begin{proof}
Assume (i). Then, on one hand, applying Proposition~\ref{p:ultraproduct}, we know that $M$ is a subset of an $\RR$-tree. And on the other hand, $M$ must be negligible as $\ell_1$ does not contain $L_1$. Thus (ii) follows.

Now assume (ii), fix $\varepsilon>0$ and let $N$ be the metric space given by Theorem~\ref{tm:separable_tree_embedding}. Then $\lipfree{M}$ is linearly $(1+\varepsilon)$-isomorphic to $\lipfree{N}$, which is linearly isometric to a subspace of $\lipfree{N \cup \cl{\Br(N)}}$. Now $\lipfree{N \cup \cl{\Br(N)}}$ is linearly isometric to $\ell_1$ by \cite[Corollary 3.4]{Godard_2010} as both $N$ and $\cl{\Br(N)}$ are negligible and closed in $\conv(N)$. Hence we get (i).
\end{proof}

As a consequence of Theorem \ref{th:lipfree_tree_embedding} we obtain the following: 

\begin{corollary}\label{c:RNPSchurL1}
Let $M$ be a closed subset of an $\RR$-tree. Then the following are equivalent:
\begin{enumerate}[label={\upshape{(\roman*)}}]
\item $\lambda(M)=0$,
\item $\lipfree{M}$ has the Schur property,
\item $\lipfree{M}$ has the Radon-Nikod\'ym property.
\item $\lipfree{M}$ does not contain an isomorphic copy of $L_1$.
\end{enumerate}
\end{corollary}

\begin{proof}
If $\lambda(M)>0$ then $\lipfree{M}$ contains $L_1$ isometrically by \cite[Corollary 3.4]{Godard_2010}, so (iv) implies (i). Clearly, (ii) or (iii) imply (iv). Now assume (i). To prove (ii) and (iii), it suffices to show that every closed separable subspace $X$ of $\lipfree{M}$ has the Schur and Radon-Nikod\'ym properties. But it is easy to see that for any such $X$ there is a closed separable set $K\subset M$ such that $X\subset\lipfree{K}$, and clearly $\lambda(K)=0$ so $X$ is isomorphic to a subspace of $\ell_1$ by Theorem \ref{th:lipfree_tree_embedding}.
\end{proof}

It is currently unknown whether any of the equivalences (ii)$\Leftrightarrow$(iii), (ii)$\Leftrightarrow$(iv), (iii)$\Leftrightarrow$(iv) hold in general for Lipschitz-free spaces.

\begin{remark}
Another equivalent condition in Corollary \ref{c:RNPSchurL1} can be stated in terms of the set $\mathrm{SNA}(M) \subset \Lip_0(M)$ of \emph{strongly norm attaining} Lipschitz functions on $M$, i.e. those that attain their Lipschitz constant between two points of $M$. By \cite[Theorem 2.3 and Theorem 3.1]{CCGMR}, we get that the conditions (i)-(iv) are also equivalent to $\cl{\mathrm{SNA}(M)}=\Lip_0(M)$.
\end{remark}

\section{Rearrangements of subsets of \texorpdfstring{$\RR$}{R}-trees}

This section is devoted to proving Theorem \ref{tm:separable_tree_embedding}. The proof will be constructive, repeatedly applying a certain procedure on the tree that ``clears'' the branching points contained in a given segment of the tree so that their measure becomes 0, while keeping the other components of the tree unmodified. For the sake of economy of language, we shall give a name to this transformation:

\begin{definition}
Let $M$ be a subset of an $\RR$-tree $T$. A \emph{rearrangement} of $(M,T)$ is a pair $(\psi,T')$ where $T'$ is an $\RR$-tree that contains $T$ and $\psi\colon T\rightarrow T'$ is a root-preserving mapping that satisfies the following:
\begin{enumerate}[label={\upshape{(\Roman*)}}]
\item\label{cond:ra:order_p} $\psi$ preserves the order on $T$,
\item\label{cond:ra:order_i} $\psi|_M$ is an order isomorphism, and
\item\label{cond:ra:ineq} there is a constant $C>0$ such that
$$
d(p,q)\leq d(\psi(p),\psi(q))\leq C\cdot d(p,q)
$$
for all $p,q\in M$.
\end{enumerate}
We will say that the rearrangement has constant $C$.
\end{definition}

It is clear that the composition of rearrangements is again a rearrangement, the constant of the result being bounded by the product of the respective constants. Condition \ref{cond:ra:ineq} shows that $\psi|_M$ is a $C$-Lipschitz homeomorphism between $M$ and $\psi(M)$, and so $\psi(M)$ is complete if $M$ is. In particular, $\psi|_M$ is continuous and injective. However, $\psi$ need not (and will not) be either, and may e.g. map different branching points into the same one, or intervals of $T$ into disconnected sets.

One consequence of the definition is that $\psi$ does not decrease the length measure of $M$:

\begin{lemma}
\label{lm:rearr_measure_inequality}
Let $M$ be a closed subset of an $\RR$-tree $T$ and let $(\psi,T')$ be a rearrangement of $(M,T)$. Then $\lambda(\psi(M))\geq \lambda(M)$.
\end{lemma}

\begin{proof}
We first claim that
\begin{equation}
\label{eq:rearr_measure_ineq}
\lambda(M\cap [x,y])\leq\lambda(\psi(M)\cap [\psi(x),\psi(y)])
\end{equation}
for any $x,y\in M$ such that $x\preccurlyeq y$. 
Indeed, property \ref{cond:ra:order_i} implies that $\psi(M \cap [x,y])=\psi(M) \cap [\psi(x),\psi(y)]$.
The desired inequality \eqref{eq:rearr_measure_ineq} thus follows from a corresponding inequality for the Lebesgue measure on~$\RR$.

Now fix $\varepsilon>0$ and let $I_k$ for $k=1,\ldots,n$ be disjoint segments in $T$ such that
$$
\lambda(M) < \sum_{k=1}^n \lambda(M\cap I_k) + \varepsilon .
$$
We may assume that the segments $I_k$ are of the form $[x,y]$ for $x\prec y$, as we may replace $[x,y]$ by $[x',x]\cup [y',y]$ where $x' \in [x\wedge y,x]$ and $y' \in [x \wedge y, y]$ are suitably chosen.
Since $M$ is closed, for every $k=1,\ldots,n$ we have $M\cap I_k=M\cap [p_k,q_k]$ for some $p_k,q_k\in M$ such that $p_k\preccurlyeq q_k$. Hence \eqref{eq:rearr_measure_ineq} implies
$$
\sum_{k=1}^n \lambda(M\cap I_k) = \sum_{k=1}^n \lambda(M\cap [p_k,q_k]) \leq \sum_{k=1}^n \lambda(\psi(M)\cap [\psi(p_k),\psi(q_k)])
$$
Condition \ref{cond:ra:order_i} implies that the sets $\psi(M)\cap [\psi(p_k),\psi(q_k)]$ are pairwise disjoint, hence we get $\lambda(M)<\lambda(\psi(M))+\varepsilon$. This completes the proof.
\end{proof}

Under the same assumptions, we actually also have that $\lambda(\psi(M))\leq C\cdot\lambda(M)$. Since we do not need this last fact in what follows, we omit its proof.

Let us now fix a separable $\RR$-tree $T$ and a closed subset $M\subset T$. For points $x\neq y\in T$, we will denote by $T_{xy}$ the union of the connected components of $T\setminus\set{x}$ that do not contain $y$. Note that $\cl{T_{xy}}=T_{xy}\cup\set{x}$, and that $T_{xy}$ is connected if and only if $x\notin\Br(T)$. We will consider a specific type of rearrangement that modifies the topology of the branching points contained within a given segment.

\begin{definition}
Let $(\psi,T')$ be a rearrangement of $(M,T)$. We will say that the rearrangement is \emph{subordinated} to an interval $[x,y]\subset T$ if it satisfies the following:
\begin{enumerate}[label={\upshape{(\Roman*)}}]
\setcounter{enumi}{3}
\item\label{cond:ra:disj} the images under $\psi$ of disjoint connected subsets of $T\setminus [x,y]$ are disjoint,
\item\label{cond:ra:isom} the restriction of $\psi$ to each connected subset of $T\setminus [x,y]$ is an isometry, and
\item\label{cond:ra:ident} the restriction of $\psi$ to $T_{xy} \cup [x,y] \cup T_{yx}$ is the identity.
\end{enumerate}
\end{definition}

The idea here is to change the positions within $[x,y]$ where the connected subtrees are attached, without modifying the subtrees themselves. In particular, the length measure of $M$ is preserved by such a rearrangement:

\begin{lemma}
\label{lm:subrearr_preserves_measure}
Let $M$ be a subset of a separable $\RR$-tree $T$ and $(\psi,T')$ be a rearrangement of $(M,T)$ subordinated to an interval. Then $\lambda(\psi(M))=\lambda(M)$.
\end{lemma}

\begin{proof}
Since $T$ is separable, $\Br(T)$ is countable and $T\setminus\set{p}$ has countably many connected components for any $p\in T$. In particular, if $(\psi,T')$ is subordinated to $[x,y]$ then $T\setminus(T_{xy}\cup [x,y]\cup T_{yx})$ has countably many connected components. Enumerate them as $A_n$, $n\in I$, where $I$ is countable. Then
$$
\lambda(M)=\lambda(M\cap(T_{xy}\cup [x,y]\cup T_{yx}))+\sum_{n\in I}\lambda(M\cap A_n) .
$$
By conditions \ref{cond:ra:disj}-\ref{cond:ra:ident}, $\psi$ is an isometry when restricted to any of these sets and their images are disjoint, hence
$$
\lambda(M)=\lambda(\psi(M\cap(T_{xy}\cup [x,y]\cup T_{yx})))+\sum_{n\in I}\lambda(\psi(M\cap A_n))
$$
which is clearly equal to $\lambda(\psi(M))$.
\end{proof}

\subsection{Construction of rearrangements}

We will now proceed to prove a series of lemmas where rearrangements are constructed so that they are subordinated to intervals of $T$ of increasing coverage.

\begin{lemma}
\label{lm:separated_interval}
Let $x\prec y \in T\setminus\Br(M)$ be such that $[x,y]\cap M=\varnothing$, and $\varepsilon>0$. Then there is a rearrangement $(\psi,T')$ of $(M,T)$ with constant $1+\varepsilon$, subordinated to $[x,y]$, such that the set $\Br(\psi(M))\cap [x,y]$ is finite.
\end{lemma}

\begin{proof}
Let $L=d(M,[x,y])\cdot\varepsilon/3$ and note that $L>0$ (since $[x,y]\cap M=\varnothing$ and $M$ is closed).
Find a finite sequence of points
$$
x=z_0\prec z_1\prec\ldots\prec z_m=y
$$
in $[x,y]$ such that $d(z_{k-1},z_k)\leq L$ and $z_k\notin \Br(T)$ for $k=1,\ldots,m$; this is possible because $\Br(T)$ is countable.

For any $b\in\Br(T)\cap (x,y)$ consider the connected components of $T\setminus\set{b}$ that contain neither $x$ nor $y$. There are at most countably many such components; enumerate them as $(A_n)_{n}$ and let $(b_n)_{n}$ be the corresponding branching points in $[x,y]$,
which we now consider as their respective roots.
Note that we possibly have $b_m = b_{n}$ if $A_m$ and $A_{n}$ share the same root.
We construct a new $\RR$-tree $T'$ as follows: for each $n$, take an isometric copy $A'_n \cup \set{b'_n}$ of $A_n \cup \set{b_n}$, add a segment $B'_n$ of length $L$ at $b'_n$, and attach the end of this segment to $z_{k_n}$, where $k_n\in\NN$ is chosen so that $z_{k_n-1}\prec b_n\prec z_{k_n}$. Now define the mapping $\psi$ as follows: for
$$
p\in T_{xy} \cup [x,y] \cup T_{yx} = T\setminus\bigcup_{n=1}^\infty A_n
$$
let $\psi(p)=p$, and for $p\in A_n$ for some $n$, let $\psi(p)=\psi_n(p)$ where $\psi_n\colon A_n\cup\set{b_n}\rightarrow A'_n\cup\set{b'_n}$ is the corresponding isometry. The effect of $\psi$ may also be described as follows: for each $p\in M$ in a component $A_n$, its distance to $[x,y]$ is increased by $L$ and its meet with $y$ is moved from $b_n$ to $z_{k_n}$. See Figure \ref{fig:rearrangement_1} for reference.

\begin{figure*}[t]
  \centering
  \includegraphics[width=0.9\textwidth]{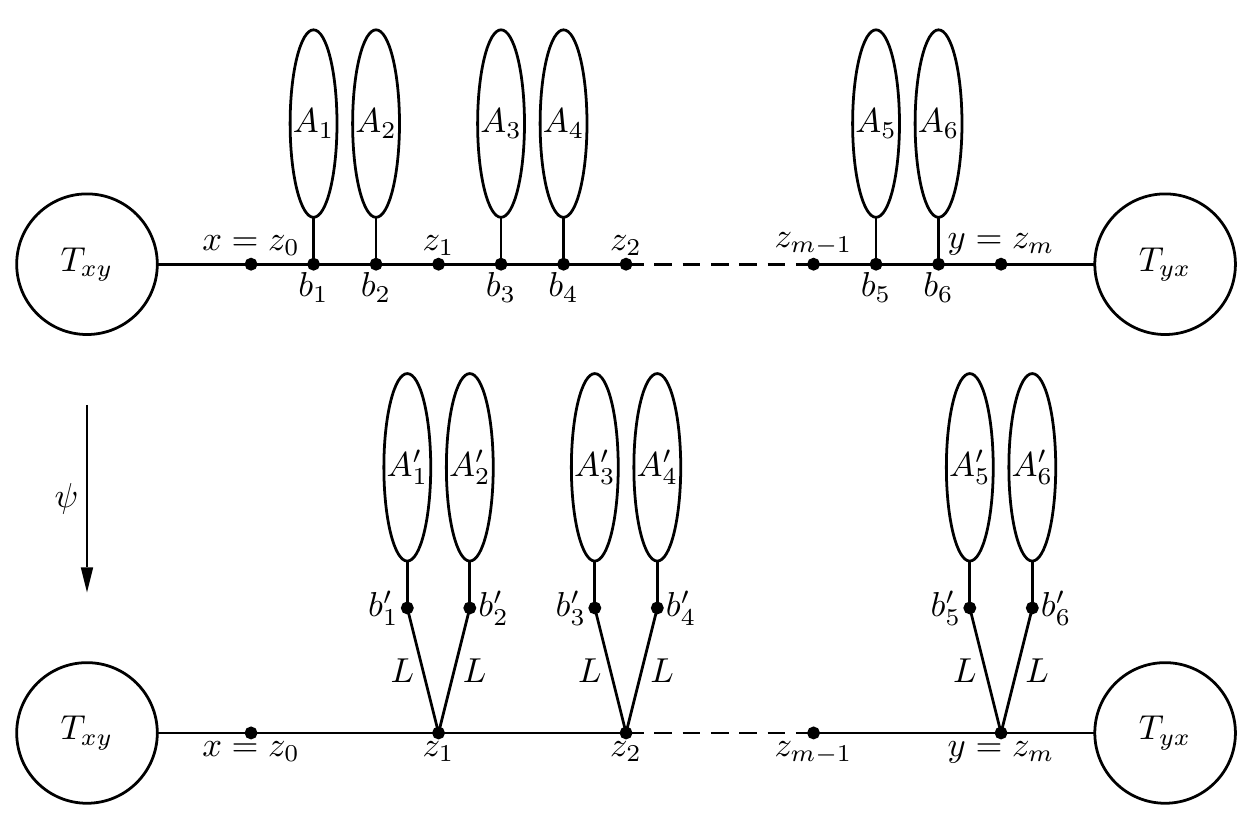}
	\caption{Representation of the construction in Lemma \ref{lm:separated_interval}.}
	\label{fig:rearrangement_1}
\end{figure*}

Let us show that $(\psi,T')$ is the desired rearrangement of $(M,T)$. Indeed, conditions \ref{cond:ra:disj}-\ref{cond:ra:ident} hold trivially. For condition \ref{cond:ra:order_p}, notice that $p\preccurlyeq q$ implies that either
\begin{itemize}
\item $p,q\in A_n$ for some $n$, in which case $\psi|_{A_n\cup\set{b_n}}=\psi_n$ is an isometry and therefore an order isomorphism,
\item $p,q\in T_{xy}\cup [x,y]\cup T_{yx}$, with a similar conclusion, or
\item $p\in T_{xy}\cup [x,y]$ and $q\in A_n$, in which case $p\preccurlyeq b_n\prec q$ and $\psi(p)=p\preccurlyeq b_n\preccurlyeq z_{k_n}\prec\psi(q)$.
\end{itemize}
Similar reasoning shows that condition \ref{cond:ra:order_i} holds, taking into account that neither $M$ nor $\psi(M)$ intersect the segment $[x,y]$. It is also clear by construction that $\psi(M)$ intersects no $A_n$ and so
$$
\Br(\psi(M))\cap [x,y]\subset\set{z_0,z_1,\ldots,z_m} .
$$

Finally, we will prove that condition \ref{cond:ra:ineq} holds with $C=1+\varepsilon$ by considering all possible pairs of points in $M$. We have already seen that $\psi|_M$ is an isometry when restricted to $T_{xy}\cup T_{yx}$ or to any $A_n$. For the remaining cases:
\begin{itemize}

\item If $a\in A_n$ and $c\in T_{xy}\cup T_{yx}$ then $d(a,c) = d(a,b_n) + d(b_n,c)$ and
\begin{align*}
d(\psi(a),\psi(c)) &= d(\psi(a),b'_n) + d(b'_n,z_{k_n}) + d(z_{k_n},c) \\
&= d(a,b_n) + L + d(z_{k_n},c) \\
&= d(a,c) + L + d(z_{k_n},c) - d(b_n,c) \\
&= d(a,c) + L \pm d(b_n,z_{k_n}) 
\end{align*}
where the sign depends on whether $c\in T_{xy}$ or $c\in T_{yx}$. Since $d(b_n,z_{k_n})<L$, we get in any case
$$
1\leq\frac{d(\psi(a),\psi(c))}{d(a,c)}\leq 1+\frac{2L}{d(a,c)}< 1+\varepsilon .
$$

\item If $a\in A_n$ and $\hat{a}\in A_{\hat{n}}$ with $k_n=k_{\hat{n}}=k$, then we have
$$
d(a,\hat{a})=d(a,b_n)+d(b_n,b_{\hat{n}})+d(b_{\hat{n}},\hat{a})
$$
and
\begin{align*}
d(\psi(a),\psi(\hat{a})) &= d(\psi(a),b'_n) + d(b'_n,z_k) + d(z_k,b'_{\hat{n}}) + d(b'_{\hat{n}},\psi(\hat{a})) \\
&= d(a,b_n) + 2L + d(b_{\hat{n}},\hat{a}) \\
&= d(a,\hat{a}) + 2L - d(b_n,b_{\hat{n}}) .
\end{align*}
Since $d(b_n,b_{\hat{n}})<L$, we obtain
$$
1\leq\frac{d(\psi(a),\psi(\hat{a}))}{d(a,\hat{a})}\leq 1+\frac{2L}{d(a,\hat{a})}< 1+\varepsilon .
$$

\item If $a\in A_n$ and $\hat{a}\in A_{\hat{n}}$ where $k_n\neq k_{\hat{n}}$, then again
$$
d(a,\hat{a})=d(a,b_n)+d(b_n,b_{\hat{n}})+d(b_{\hat{n}},\hat{a})
$$
and
\begin{align*}
d(\psi(a),\psi(\hat{a})) &= d(\psi(a),b'_n) + d(b'_n,z_{k_n}) + d(z_{k_n},z_{k_{\hat{n}}}) \\
&\qquad\qquad + d(z_{k_{\hat{n}}},b'_{\hat{n}}) + d(b'_{\hat{n}},\psi(\hat{a})) \\
&= d(a,b_n) + L + d(z_{k_n},z_{k_{\hat{n}}}) + L + d(b_{\hat{n}},\hat{a}) \\
&= d(a,\hat{a}) + 2L + d(z_{k_n},z_{k_{\hat{n}}}) - d(b_n,b_{\hat{n}}) \\
&= d(a,\hat{a}) + 2L \pm \pare{d(b_n,z_{k_n}) - d(b_{\hat{n}},z_{k_{\hat{n}}})}
\end{align*}
where the sign depends on which of $k_n$, $k_{\hat{n}}$ is greater. In any case
$$
1\leq\frac{d(\psi(a),\psi(\hat{a}))}{d(a,\hat{a})}\leq 1+\frac{3L}{d(a,\hat{a})}\leq 1+\varepsilon .
$$

\end{itemize}
This covers all cases and ends the proof.
\end{proof}

In the following lemma we will deal with the situation when $x$ and $y$ are allowed to be elements of $M$.

\begin{lemma}
\label{lm:open_interval}
Let $x\prec y\in M$ be such that $(x,y)\cap M=\varnothing$, and $\varepsilon>0$. Then there is a rearrangement $(\psi,T')$ of $(M,T)$ with constant $1+\varepsilon$, subordinated to $[x,y]$, such that the set $\Br(\psi(M))\cap (x,y)$ is countable and its accumulation points are contained in $\set{x,y}$.
\end{lemma}

\begin{proof}
Choose a doubly infinite sequence $(z_k)_{k\in\ZZ}$ of elements in $(x,y)$ such that $z_k\prec z_{k'}$ if $k<k'$, $\lim_{k\rightarrow-\infty} z_k=x$, and $\lim_{k\rightarrow\infty} z_k=y$. Since $\Br(T)$ is countable, we may choose them so that $\Br(T)$ does not intersect $Z=\set{z_k:k\in\ZZ}$. Then $(x,y)\setminus Z$ is the disjoint union of the open intervals $I_k^o=(z_{k-1},z_k)$ for $k\in\ZZ$. Express $T\setminus Z$ as a partition into connected components
$$
T\setminus Z = A_{-\infty} \cup \pare{\bigcup_{k=-\infty}^\infty A_k} \cup A_{\infty}
$$
where $I_k^o\subset A_k$ for $k\in\ZZ$ and we denote $A_{-\infty}=\cl{T_{xy}}$, $A_{\infty}=\cl{T_{yx}}$, $z_{-\infty}=x$ and $z_{\infty}=y$.

For each $k\in\ZZ$, let $I_k=[z_{k-1},z_k]$ and
$$
\varepsilon_k=\varepsilon\cdot\min\set{\frac{d_k}{d_k+d(z_{k-1},x)},\frac{d_k}{d_k+d(z_k,y)}}<\varepsilon
$$
where $d_k=d(M,I_k)>0$, and use Lemma \ref{lm:separated_interval} to obtain a rearrangement $(\psi_k,T_k)$ of $(M,T)$ subordinated to $I_k$ with constant $1+\varepsilon_k$, such that $\Br(\psi_k(M))\cap I_k$ is finite.
By Fact~\ref{fact:union_completion}, we may assume that each $T_k$ is complete.
Notice that, since $(\psi_k,T_k)$ is a rearrangement subordinated to $I_k$, point \ref{cond:ra:ident} implies $T_{xy} \cup [x,y] \cup T_{yz}$ is a subset of $T_k$ for every $k$. 
Thus, we may define an $\RR$-tree metric on  $T'=\bigcup_{k\in\ZZ}T_k$ as in Fact~\ref{fact:union}
and define $\psi\colon T\rightarrow T'$ by
$$
\psi(p)=\begin{cases}
\psi_k(p) & \text{if $p\in A_k$ for some $k\in\ZZ$} \\
p & \text{otherwise}
\end{cases} .
$$
We claim that $(\psi,T')$ is the desired rearrangement. Indeed, since every $\psi_k$ restricts to the identity outside of $A_k$, conditions \ref{cond:ra:disj}-\ref{cond:ra:ident} follow from the corresponding conditions for the $\psi_k$. Notice also that
$$
\Br(\psi(M))\cap (x,y)=\bigcup_{k=-\infty}^\infty\Br(\psi(M))\cap I_k=\bigcup_{k=-\infty}^\infty\Br(\psi_k(M))\cap I_k
$$
is a countable union of finite sets that has no accumulation points other than (possibly) $x$ and $y$.

To check condition \ref{cond:ra:ineq}, let $p\in M\cap A_m$ and $q\in M\cap A_n$ where $m,n\in\ZZ\cup\set{-\infty,\infty}$. If $m=n$ then the inequalities follow from $\varepsilon_n\leq\varepsilon$ and the properties of $\psi_n$ (or from isometry, if $n=\pm\infty$). Otherwise suppose $m<n$.
Then we have
\begin{align*}
d(p,q) &= d(p,z_m) + d(z_m,z_{n-1}) + d(z_{n-1},q) \\
&= d(p,y) - d(z_m,y) + d(z_m,z_{n-1}) + d(x,q) - d(x,z_{n-1})
\end{align*}
and similarly
\begin{multline*}
d(\psi(p),\psi(q)) = d(\psi(p),y) - d(z_m,y) \\ + d(z_m,z_{n-1}) + d(x,\psi(q)) - d(x,z_{n-1})
\end{multline*}
hence
$$
d(\psi(p),\psi(q)) - d(p,q) = d(\psi_m(p),y) - d(p,y) + d(\psi_n(q),x) - d(q,x) .
$$
Since $\psi_m$ and $\psi_n$ are rearrangements and $x,y\in M$, this quantity is between $0$ and $\varepsilon_m d(p,y) + \varepsilon_n d(q,x)$. Now notice that
$$
\varepsilon_m d(p,y)\leq\varepsilon\,\frac{d_m}{d_m+d(z_m,y)}(d(p,z_m)+d(z_m,y))\leq\varepsilon d(p,z_m)
$$
and similarly $\varepsilon_n d(q,x)\leq\varepsilon d(q,z_{n-1})$, so we get
$$
d(\psi(p),\psi(q))\leq d(p,q)+\varepsilon d(p,z_m)+\varepsilon d(z_{n-1},q)\leq (1+\varepsilon)d(p,q)
$$
as required.

For condition \ref{cond:ra:order_p}, suppose similarly that $p\in A_m$ and $q\in A_n$ with $m<n$ are such that $p\preccurlyeq q$. Then $p\preccurlyeq z_m\preccurlyeq z_{n-1}\preccurlyeq q$ and it follows
$$
\psi(p)\preccurlyeq\psi(z_m)=z_m\preccurlyeq z_{n-1}=\psi(z_{n-1})\preccurlyeq\psi(q)
$$
where the first and last inequalities are implied by the properties of $\psi_m$ and $\psi_n$. The cases where $p$ or $q$ are in $Z$ are handled similarly. Finally, if $p,q\in M$ and $\psi(p)\preccurlyeq\psi(q)$, then either both are in the same $A_n$, in which case $p\preccurlyeq q$ follows from order isomorphism of $\psi_n$, or $\psi(p)\in A_{-\infty}$ so that $p=\psi(p)$ and $p\preccurlyeq q$ follows easily. This proves condition \ref{cond:ra:order_i}.
\end{proof}

Finally, we extend the construction to arbitrary points $x,y$ of $T$:

\begin{lemma}
\label{lm:closed_interval}
Let $x\prec y\in T$, and $\varepsilon>0$. Then there is a rearrangement $(\psi,T')$ of $(M,T)$ with constant $1+\varepsilon$, subordinated to $[x,y]$, such that the set
$$
(\Br(\psi(M))\setminus\psi(M))\cap (x,y)
$$
is countable and its accumulation points are contained in $\psi(M)\cup\set{x,y}$.
\end{lemma}

\begin{proof}
Since $M$ is closed, $(x,y)\setminus M$ is the disjoint union of at most countably many open intervals $I_k^o=(x_k,y_k)$, $k\in\NN$, where $x_k,y_k\in M$ and $x_k\prec y_k$. For each $k\in\NN$, use Lemma~\ref{lm:open_interval} to obtain a rearrangement $(\psi_k,T_k)$ of $(M,T)$ subordinated to $I_k=[x_k,y_k]$ and with constant $1+\varepsilon$, such that $\Br(\psi_k(M))\cap I_k^o$ is countable and discrete.
As before, $T_k$ can be taken complete and $T_{xy} \cup[x,y] \cup T_{yx} \subset T_k$ for every $k$. 
Thus, we may define an $\RR$-tree metric on $T'=\bigcup_{k\in\NN}T_k$ as in Fact~\ref{fact:union}
and define $\psi\colon T\rightarrow T'$ by
$$
\psi(p)=\begin{cases}
\psi_k(p) & \text{if $p\in A_k$ for some $k\in\NN$} \\
p & \text{otherwise}
\end{cases}
$$
where $A_k$ is the connected component of $T\setminus (M\cap [x,y])$ that contains $I_k^o$. Then $(\psi,T')$ is a rearrangement of $(M,T)$ subordinated to $[x,y]$, as can be proven mimicking the arguments in the proof of Lemma~\ref{lm:open_interval}. Moreover, given two points $p,q\in M$ in different components of $T\setminus (M\cap [x,y])$ there is always $z\in M\cap (p,q)\cap (\psi(p),\psi(q))$, and one gets that $\psi$ has constant $1+\varepsilon$ by applying the corresponding inequalities to $d(\psi(p),z)$ and $d(\psi(q),z)$. It is also clear that
\begin{align*}
(\Br(\psi(M))\setminus\psi(M))\cap (x,y) &= \Br(\psi(M))\cap ((x,y)\setminus M) \\
&= \bigcup_k \Br(\psi_k(M))\cap I_k^o
\end{align*}
is a countable set and that all of its accumulation points are contained in $\psi(M)\cup\set{x,y}$.
\end{proof}

\subsection{Proof of Theorem \ref{tm:separable_tree_embedding}}

We are now ready to finish the promised proof. For convenience, we restate the theorem here.

\begin{T1}
Let $M$ be a complete separable metric space that is a subset of an $\RR$-tree. Then, for every $\varepsilon>0$, $M$ is $(1+\varepsilon)$-Lipschitz homeomorphic to a subset $N$ of an $\RR$-tree such that $\lambda(N)=\lambda(M)$ and $\lambda(\cl{\Br(N)}\setminus N)=0$, where the closure is taken in $\conv(N)$.
\end{T1}

\begin{proof}
Let $T=\conv(M)$ and choose an element $0\in M$ as its root. Let $(\xi_n)_{n=1}^\infty$ be a dense sequence in $M$, and let $(x_n)_n$ be the subsequence obtained by eliminating all elements $\xi_n$ such that $\xi_n\preccurlyeq\xi_k$ for some $k<n$. We may assume that $(x_n)_n$ is an infinite sequence, otherwise $M$ is compact and $\lambda(\cl{\Br(M)}\setminus M)=0$ by the argument in \cite[Lemma 7]{DaKaPr_2016} so there is nothing to prove. Denote
$$
Q_n=\conv\set{0,x_1,\ldots,x_n}=\bigcup_{k=1}^n [0,x_k]
$$
so that $Q_n\subsetneq Q_{n+1}$ for all $n$. Let
$$
Q=\bigcup_{n=1}^\infty Q_n=\conv(\set{0}\cup\set{x_n:n\in\NN})
$$
and $M'=M\cap Q$. Notice that $M'$ contains all of $M$ except possibly for some leaves of $T$ that are accumulation points of $M'$, hence $Q\cup(M\setminus M')=T$ and $M'$ is dense in $M$. Notice also that $Q=\conv(M')$ and so $\Br(Q)=\Br(M')=\Br(M)$ using Fact \ref{fact:dense_br}.

Choose a sequence $(\varepsilon_n)_{n=1}^\infty$ of strictly positive numbers such that
$$
\prod_{n=1}^\infty (1+\varepsilon_n)\leq 1+\varepsilon .
$$
Apply Lemma \ref{lm:closed_interval} to obtain a rearrangement $(\psi_1,T_1)$ of $(M,T)$ with constant $1+\varepsilon_1$, subordinated to $I_1=[0,x_1]$, such that the set
$$
B_1 = (\Br(\psi_1(M))\setminus\psi_1(M))\cap I_1^o
$$
is countable and each of its accumulation points is either $0$, $x_1$ or an element of $\psi_1(M)$. Now define inductively for each $n\in\NN$
$$
b_n=\max_{1\leq k\leq n}\pare{\Psi_n(x_{n+1})\wedge\Psi_n(x_k)}
$$
where $\Psi_n=\psi_n\circ\ldots\circ\psi_1$ (notice that all elements lie in $[0,\Psi_n(x_{n+1})]$, so it makes sense to take the maximum), and use Lemma \ref{lm:closed_interval} to construct a rearrangement $(\psi_{n+1},T_{n+1})$ of $(\Psi_n(M),T_n)$ that has constant $1+\varepsilon_{n+1}$, is subordinated to the segment
$$
I_{n+1}=[b_n,\Psi_n(x_{n+1})]
$$
and such that the set
$$
B_{n+1} = (\Br(\Psi_{n+1}(M))\setminus\Psi_{n+1}(M))\cap I_{n+1}^o
$$
is countable and each of its accumulation points is either $b_n$, $\Psi_n(x_{n+1})$ or an element of $\Psi_{n+1}(M)$. 
Let us also follow the convention that $b_0=0$ and $\Psi_0$ is the identity on $T$. Note that $\lambda(\Psi_n(M))=\lambda(M)$ for all $n\in\NN$ by Lemma \ref{lm:subrearr_preserves_measure} and induction.

We now claim the following:

\begin{claim}
\label{cl:claim1}
If $m>n$ then $\psi_m$ restricts to the identity on $\bigcup_{k=1}^n [0,\Psi_n(x_k)]$.
\end{claim}

\begin{proof}[Proof of Claim \ref{cl:claim1}]
By induction, it is enough to show that $\psi_{n+1}$ restricts to the identity on $[0,\Psi_n(x_n)]$. Let $U$ be the component of $T_n\setminus I_{n+1}^o$ that contains $0$; since $\psi_{n+1}$ is subordinated to $I_{n+1}$, it will suffice to check that $[0,\Psi_n(x_n)]\subset U$, i.e. that $\Psi_n(x_n)\in U$. To see this, notice that $b_n\in U$ as $b_n\preccurlyeq\Psi_n(x_{n+1})$, and moreover $\Psi_n(x_{n+1})\wedge\Psi_n(x_n)\preccurlyeq b_n$ by definition, therefore $[b_n,\Psi_n(x_n)]\subset U$.
\end{proof}

For every $p\in Q$, define $\Psi(p)=\Psi_n(p)$ where $n\in\NN$ is such that $p\in Q_n$. Claim \ref{cl:claim1} ensures that this definition is independent of the choice of $n$. Let us observe that
$$
\Psi(x_n)=\Psi_n(x_n)=\Psi_{n-1}(x_n)
$$
for any $n\in\NN$. Indeed, the first equality follows from $x_n\in Q_n$, and the second one from the fact that $\psi_n$ is subordinated to a segment containing $\Psi_{n-1}(x_n)$. Note also that all the $\Psi(x_n)$ are different, since the restriction of each $\Psi_n$ to $M$ is injective.

Let $T'$ be the completion of $\bigcup_{n=1}^\infty T_n$, which is an $\RR$-tree by Fact~\ref{fact:union_completion}. It is easy to see that $(\Psi,T')$ is a rearrangement of $(M',Q)$ with constant $1+\varepsilon$. Indeed, let $p,q\in Q$, then we may find $n\in\NN$ such that $p,q\in Q_n$ and therefore $\Psi(p)=\Psi_n(p)$ and $\Psi(q)=\Psi_n(q)$. Since every $\Psi_n$ preserves the order and is an order isomorphism when restricted to $M$, the same is true for $\Psi$. And if $p,q\in M'$, then we have
$$
d(p,q)\leq d(\Psi_n(p),\Psi_n(q))\leq d(p,q)\cdot\prod_{k=1}^n (1+\varepsilon_k)\leq (1+\varepsilon)d(p,q)
$$
for every $n\in\NN$. Now extend $\Psi|_{M'}$ continuously to a mapping $\psi\colon M\rightarrow T'$, and define $\Psi(p)=\psi(p)$ for $p\in M\setminus M'$, then it is clear that $(\Psi,T')$ is a rearrangement of $(M,T)$ with constant $1+\varepsilon$.

Let $N=\Psi(M)=\cl{\Psi(M')}$. Then $N$ is $(1+\varepsilon)$-Lipschitz homeomorphic to $M$ and Lemma \ref{lm:rearr_measure_inequality} implies that $\lambda(N)\geq\lambda(M)$. To prove the reverse inequality, let $\delta>0$ and $J_k$, $k=1,\ldots,n$ be disjoint segments in $T'$ such that
$$
\lambda(N) < \sum_{k=1}^n \lambda(N\cap J_k) + \delta .
$$
As in the proof of Lemma \ref{lm:rearr_measure_inequality}, we may assume that $J_k=[\Psi(p_k),\Psi(q_k)]$ for $p_k,q_k\in M$ and $p_k\preccurlyeq q_k$. In fact, we may assume that $p_k,q_k\in M'$ since $M\setminus M'$ consists of leaves of $T$. Therefore there is $n_0\in\NN$ such that $p_k,q_k\in Q_{n_0}$ for all $k=1,\ldots,n$. This implies that $N\cap J_k\subset\Psi_{n_0}(M)$: indeed, if $a\in N\cap J_k$ then $a=\psi(z)$ for some $z\in M$ and \ref{cond:ra:order_i} implies that $z\in [p_k,q_k]\subset Q_{n_0}$. Hence
$$
\sum_{k=1}^n \lambda(N\cap J_k) \leq \sum_{k=1}^n \lambda(\Psi_{n_0}(M)\cap J_k) \leq \lambda(\Psi_{n_0}(M)) = \lambda(M)
$$
so $\lambda(N)<\lambda(M)+\delta$. Since $\delta$ was arbitrary, $\lambda(N)\leq\lambda(M)$ follows.

To complete the proof of the theorem, it only remains to be shown that $\lambda(\cl{\Br(N)}\setminus N)=0$. In order to do that, let us denote $B=\bigcup_{n=1}^\infty B_n$. We will prove the following statements:

\begin{claim}
\label{cl:claim2}
If $m>n$ then $I_m^o \cap [0,\Psi(x_n)]=\varnothing$.
\end{claim}

\begin{claim}
\label{cl:claim3}
$\Br(\Psi(M'))\cap [0,\Psi(x_n)] \subset \Psi(M')\cup B_1\cup\ldots\cup B_n$.
\end{claim}

\begin{claim}
\label{cl:claim4}
$(\cl{B}\setminus N) \cap \Psi(Q)$ is a countable set.
\end{claim}

Using these claims, we finish our proof as follows. Since any element of $\Br(\Psi(M'))$ necessarily belongs to some $[0,\Psi(x_n)]$, Claim \ref{cl:claim3} implies that $\Br(\Psi(M'))\subset\Psi(M')\cup B$. By Fact \ref{fact:dense_br} we have $\Br(N)=\Br(\Psi(M'))$, and so we get $\cl{\Br(N)}\subset N\cup\cl{B}$. Thus, it is enough to show that $\cl{B}\setminus N$ is a negligible subset of $T'$. Claim~\ref{cl:claim4} shows that the intersection of $\cl{B}\setminus N$ with any segment of the form $[0,x]$, $x\in\Psi(M')$ is negligible, hence also for $x\in N$. It follows that $\lambda(\cl{B}\setminus N)=0$ and this completes the proof of the theorem.
\end{proof}

\begin{proof}[Proof of Claim \ref{cl:claim2}]
Suppose $p\in I_m^o$ is such that $p\preccurlyeq\Psi(x_n)$. Since $p\prec\Psi(x_m)$ we get $p\preccurlyeq\Psi(x_m)\wedge\Psi(x_n)\preccurlyeq b_{m-1}$, but this contradicts $b_{m-1}\prec p$.
\end{proof}

\begin{proof}[Proof of Claim \ref{cl:claim3}]
Suppose $p\in\Br(\Psi(M'))\setminus\Psi(M')$ is such that $p\preccurlyeq\Psi(x_n)$. Then $p=\Psi(x_n)\wedge\Psi(q)$ for some $q\in M'$. Since $q\in Q$, we have $p=\Psi(x_n)\wedge\Psi(x_m)$ for some $m\in\NN$ such that $q\preccurlyeq x_m$. Now let $i,j$ be chosen to minimize the value of $i$ among all representations of $p$ of the form $p=\Psi(x_i)\wedge\Psi(x_j)$ where $i<j$. We will show that $p\in B_i$. This will prove the claim, as obviously $i\leq n$.

First, let us see that $p\in I_i^o$. Indeed, this is obvious for $i=1$, and for $i>1$ the contrary would imply that
$$
p\preccurlyeq b_{i-1}=\Psi(x_i)\wedge\Psi(x_k)\preccurlyeq\Psi(x_k)
$$
for some $k<i$, and so
$$
p=\Psi(x_i)\wedge\Psi(x_j)=b_{i-1}\wedge\Psi(x_j)=\Psi(x_k)\wedge\Psi(x_j)
$$
contradicting the minimality of $i$.

Now notice that $p\in\Br(\Psi_j(M))$. By Claim \ref{cl:claim1}, the map $\phi=\psi_j\circ\psi_{j-1}\circ\ldots\circ\psi_{i+1}$ restricts to the identity on $[0,\Psi(x_i)]$ and thus $\phi(p)=p$. Since $\phi$ preserves the order, and moreover $p\notin I_k^o$ for $i<k\leq j$ by Claim \ref{cl:claim2}, we have $p=\Psi_i(x_j)\wedge\Psi_i(x_i)\in\Br(\Psi_i(M))$. But $p\in I_i^o$ so we have either $p\in B_i$ or $p\in\Psi_i(M)$, and the latter is excluded because it implies $p\in\Psi(M')$. This finishes the proof.
\end{proof}

\begin{proof}[Proof of Claim \ref{cl:claim4}]
Let $k\in\NN$ and $q\in\cl{B}\cap [0,\Psi(x_k)]$. Then there is a sequence $(q_n)_{n=1}^\infty$ in $B$ that converges to $q$, therefore $q_n\wedge\Psi(x_k)$ converges to $q\wedge\Psi(x_k)=q$. At least one of the following three options must hold:
\begin{itemize}
\item $q=\Psi(x_k)\in\Psi(M')\cap [0,\Psi(x_k)]$.
\item We may choose a subsequence $(q_{n_i})$ of $(q_n)$ such that $q_{n_i}\in [0,\Psi(x_k)]$ for all $i$. Then we have $q\in\cl{(B_1\cup\ldots\cup B_k)\cap [0,\Psi(x_k)]}$ by Claim \ref{cl:claim3}. 
\item We may choose a subsequence $(q_{n_i})$ of $(q_n)$ such that $q_{n_i}\notin [0,\Psi(x_k)]$ and $q_{n_i}\wedge\Psi(x_k)\neq\Psi(x_k)$ for all $i$. Then $q_{n_i}\wedge\Psi(x_k)\in\Br(\Psi(Q))$ and thus
$$
q\in\cl{\Br(\Psi(Q))\cap [0,\Psi(x_k)]}\subset\cl{\Br(\Psi(M'))\cap [0,\Psi(x_k)]}
$$
where we use Fact \ref{fact:order_preservation}.
\end{itemize}
Using Claim \ref{cl:claim3}, we get that
$$
\cl{B}\cap [0,\Psi(x_k)] \subset N \cup \bigcup_{n=1}^k \cl{B_n \cap [0,\Psi(x_k)]} .
$$
which covers all three cases. By construction, $\cl{B_n}$ is the union of a countable set and a subset of $\Psi_n(M)$. Hence $(\cl{B}\setminus N)\cap [0,\Psi(x_k)]$ is countable. The claim now follows since
$$
(\cl{B}\setminus N)\cap\Psi(Q)\subset\bigcup_{k=1}^\infty\big( (\cl{B}\setminus N)\cap [0,\Psi(x_k)] \big)
$$
by Fact \ref{fact:order_preservation}.
\end{proof}

\section{The proper case and extremal structure}

Now we turn to the study of the linear structure of $\F M$ when $M$ is a \emph{proper} subset of an $\RR$-tree with length measure 0. To this aim, we need to introduce the space of little Lipschitz functions.
In the literature, there are conflicting definitions of the little Lipschitz spaces. Here we choose to follow the book by Weaver (see Chapter 4 in \cite{Weaver2}) and then we will comment on the links there are with other definitions.

\begin{definition}
Let $(M,d)$ be a pointed metric space and let  $f\in\Lip_0(M)$. We will say that $f$ is 
\begin{itemize}
    \item \emph{locally flat} if for every $p\in M$ and every $\varepsilon >0$, there exists $\delta >0$ such that
    $$x,y \in  B(p,r) \quad \implies \quad |f(x) - f(y)| \leq \varepsilon d(x,y).$$
In other words, $\lim_{r\rightarrow 0}\lipnorm{f|_{B(p,r)}}=0$ for every $p\in M$.

    \item \emph{uniformly locally flat} if for every $\varepsilon >0$, there exists $\delta >0$ such that
    $$ d(x,y) \leq \delta \quad \implies \quad |f(x) - f(y)| \leq \varepsilon d(x,y).$$

    \item \emph{flat at infinity} if for every 	$\varepsilon > 0$ there exists a compact set $K \subset M$ such that
$$x,y \not\in K \quad \implies \quad  |f(x) - f(y)| \leq \varepsilon d(x,y).$$
\end{itemize}
\end{definition}
Note that if $M$ is proper, then we may replace the compact set $K$ in the last statement by a ball $B(0,r)$ of some radius $r>0$. More precisely, for a proper metric space $M$, $f\in \Lip_0(M)$ is flat at infinity if $\lim_{r\rightarrow\infty}\lipnorm{f|_{M\setminus B(0,r)}}=0$. We now introduce the so called space of little Lipschitz functions.

\begin{definition}
Let $\lip_0(M)$ be the subspace of all functions in $\Lip_0(M)$ that are uniformly locally flat and flat at infinity. 
\end{definition}

It follows from \cite[Lemma 4.16]{Weaver2} that if $f$ is flat at infinity then it is uniformly locally flat if and only if it is locally flat. Note that every $f \in \Lip_0(M)$ is flat at infinity when $M$ is compact, hence $\lip_0(M)$ consists of the locally flat elements of $\Lip_0(M)$ in that case, which is consistent with the notation used elsewhere. In other references e.g. \cite{Dalet_2015_2,GaPeRu_2017,Petitjean_2017}, the space $\lip_0(M)$ is denoted $S_0(M)$ (while $\lip_0(M)$ encompasses just those elements of $\Lip_0(M)$ that are uniformly locally flat). In fact, the definition that these authors give for $S_0(M)$ slightly differs from ours in full generality, but it coincides whenever $M$ is proper (see \cite[Lemma 4.18]{Weaver2}).

\smallskip

In \cite[Theorem 3.8]{Dalet_2015_2}, Dalet proves that $\lip_0(M)$ is an isometric predual of $\lipfree{M}$ whenever $M$ is proper and ultrametric, i.e. it satisfies the strong triangle inequality
$$
d(x,z)\leq\max\set{d(x,y),d(y,z)}
$$
for every $x,y,z\in M$. It is immediate that every ultrametric space $M$ satisfies the four point condition \eqref{eq:4pc} and that every metric segment in $M$ is trivial, i.e. only contains the endpoints, so $M$ is a subset of an $\RR$-tree such that $\lambda(M)=0$. The following theorem can therefore be regarded as a generalization of Dalet's result.

\begin{theorem}
\label{th:proper_trees}
Let $M$ be an infinite proper metric space that is a subset of an $\RR$-tree. Then the following are equivalent:
\begin{enumerate}[label={\upshape{(\roman*)}}]
\item $\lambda(M)=0$,
\item $\lipfree{M}$ is isomorphic to $\ell_1$,
\item $\lipfree{M}$ is a dual space,
\item $\lipfree{M}=\dual{\lip_0(M)}$,
\end{enumerate}
and if they hold, then $\lip_0(M)$ is isomorphic to $c_0$. If $M$ is compact, then the following condition is also equivalent:
\begin{enumerate}[label={\upshape{(\roman*)}}]
\setcounter{enumi}{4}
\item $\lipfree{M}$ is isometric to a subspace of $\ell_1$.
\end{enumerate}
\end{theorem}

Under the same assumptions, it is clear that conditions (i)--(iv) are also equivalent to any of the following ones:
\begin{enumerate}
    \item[(vi)] $\lipfree{M}$ has the Schur property,
    \item[(vii)] $\lipfree{M}$ has the Radon-Nikod\'ym property,
    \item[(viii)] $\lipfree{M}$ does not contain $L_1$.
\end{enumerate}
The equivalence of (i), (vi), (vii) and (viii) stays true even if we remove the assumption of properness, see Corollary~\ref{c:RNPSchurL1}.
When $M$ is finite all of the above properties are trivially satisfied, replacing $c_0$ and $\ell_1$ by their finite-dimensional counterparts.

\begin{proof}[Proof of Theorem \ref{th:proper_trees}]
Let $T$ be a separable $\RR$-tree containing $M$.

(i)$\Rightarrow$(iv): According to arguments in \cite{Dalet_2015_2} or \cite[Theorem 4.38]{Weaver2}, we only need to show that $\lip_0(M)$ separates points of $M$ 1-uniformly, that is, given $x\neq y \in M$ and $\varepsilon>0$, we will find $f\in\lip_0(M)$ such that $\lipnorm{f}\leq 1$ and $f(y)-f(x)\geq d(x,y)-\varepsilon$.

Let $I$ be the segment $[x,y]\in T$.
It is clear that $\phi_{xy}(I\setminus M)$ is the union of a (possibly finite) sequence of disjoint open subintervals $(I_n)_{n=1}^\infty$ of $[0,d(x,y)]$, and we have $\sum_{n=1}^\infty \lambda(I_n) = d(x,y)$.
Let $N$ be such that $\sum_{n=N+1}^\infty \lambda(I_n)\leq\varepsilon$.
We can now assume that $I_n=(a_n,b_n)$ for $n\leq N$ and that $a_1<b_1\leq a_2<\ldots<b_N$.
Define $g\colon\RR\rightarrow\RR$ by
$$
g=\sum_{n=1}^N \pare{\sum_{k=1}^n d(a_k,b_k)} \mathbf 1_{[b_n,a_{n+1}]}
$$
with the convention $a_{N+1}=d(x,y)$.
It is easy to see that the restriction of $g$ to $\phi_{xy}(I\cap M)$ is 1-Lipschitz and locally constant.
Extend this restriction to $h\colon [0,d(x,y)]\rightarrow\RR$ with $\lipnorm{h}=1$.

Now let $f$ be the restriction of $h\circ\phi_{xy}\circ\pi$ to $M$, where $\pi\colon T\rightarrow I$ is the metric projection onto $I$ (see Fact \ref{fact:metricr_projection}). Then $\lipnorm{f}=1$ and
\begin{eqnarray*}
f(y)-f(x) = \sum_{k=1}^N d(a_k,b_k)  &=&  \sum_{n=1}^{\infty} \lambda(I_n) - \sum_{n=N+1}^{\infty} \lambda(I_n) \\
&\geq&  d(x,y)-\varepsilon.
\end{eqnarray*}
Moreover, it is clear that $f$ is locally constant at every $p\in M$, so it is locally flat.
Finally we check that $f$ is flat at infinity.
Let $r>0$ and $K=\set{p\in M:d(p,I)\leq r}$, and suppose that $p,q\in M\setminus K$.
If $p,q$ lie on the same connected component of $T\setminus I$, then $f(p)=f(q)$.
Otherwise, $d(p,q)\geq 2r$ and so
$$
\frac{\abs{f(p)-f(q)}}{d(p,q)}\leq\frac{d(x,y)}{2r} \,.
$$
So, subtracting a constant if necessary, we get $f\in\lip_0(M)$ and this ends the proof.

(iv)$\Rightarrow$(iii): This is trivial.

(iii)$\Rightarrow$(ii): 
We use a variation of the argument in \cite{Dalet_2015_2}. By a result in \cite{Matousek_1990}, vector-valued Lipschitz mappings on $M$ may be extended to $T$ while increasing their Lipschitz constant by a universal factor. 
Use this to extend the isometric embedding $\delta_M\colon M\rightarrow\lipfree{M}$ to a Lipschitz mapping $f\colon T\rightarrow\lipfree{M}$, then apply the universal property of Lipschitz-free spaces \cite[Theorem 3.6]{Weaver2} to obtain an operator $F\colon\lipfree{T}\rightarrow\lipfree{M}$ such that $F\circ\delta_T=f$. Then $F(\delta_M(x))=F(\delta_T(x))=f(x)=\delta_M(x)$ for all $x\in M$, so $F$ is a projection onto $\lipfree{M}$. This shows that $\lipfree{M}$ is complemented in $\lipfree{T}$, which is isometric to $L_1(T)$ by \cite[Corollary 3.3]{Godard_2010}. We finish by applying \cite[Theorem 2]{LeSt_1973}, which states that if a complemented subspace of an $L_1$ space is a separable dual then it must be isomorphic to $\ell_1$.

(ii)$\Rightarrow$(i) and (v)$\Rightarrow$(i): Suppose that $\lambda(M)>0$. Then $\lipfree{M}$ is isomorphic to $L_1$ by \cite[Corollary 3.4]{Godard_2010}, so it cannot be isomorphic to a subspace of $\ell_1$.

If $M$ is compact then the implication (i)$\Rightarrow$(v) is contained in \cite[Proposition 8]{DaKaPr_2016}.

Finally, notice that, since $M$ is proper, $\lip_0(M)$ is isomorphic to a subspace of $c_0$ by \cite[Lemma 3.9]{Dalet_2015_2} (we remark that a correct proof of this lemma appears in \cite{Dalet_2015_arxiv}). If conditions (ii) and (iv) hold then $\dual{\lip_0(M)}$ is isomorphic to $\ell_1$, so $\lip_0(M)$ is a $\mathscr{L}_\infty$ space and the results in \cite{JoZi_1972} imply that it is actually isomorphic to $c_0$.
\end{proof}

We conclude this section by characterizing the extreme points of the ball of $\lipfree{M}$ when $M$ is a subset of an $\RR$-tree. In \cite[Theorem 2]{KaFo_1983}, Kadets and Fonf proved that if a Banach space $X$ is isometric to a subspace of $\ell_1$ then every extreme point of $\ball{X}$ is strongly extreme (or MLUR, using their notation) and hence preserved \cite{GuMoZi_2014}. Here we adapt this theorem to subspaces of $L_1(\mu)$. We will use the following fact \cite[Proposition 9.1]{GuMoZi_2014}: $x$ is a preserved extreme point of $B_X$ if and only if given two sequences $(y_n)_{n=1}^\infty$ and $(z_n)_{n=1}^\infty$ in $B_X$ such that $\frac{1}{2}(y_n+z_n)\rightarrow x$ one must have $y_n\wconv x$.

\begin{theorem}\label{t:MLURallMeasures}
Let $(\Omega,\Sigma,\mu)$ be a 
measure space and $X$ be a subspace of $L_1(\mu)$. Then every extreme point of $B_X$ is preserved.
\end{theorem}

Notice that the theorem can be applied directly to $C(K)^*$ where $K$ is a compact Hausdorff space, since this is also an $L_1$-space as is well known (see \cite{AlKa}).

\begin{proof}
Suppose $f$ is an extreme point of $B_X$ but it is not preserved.
Then there exist sequences $(g_n)_{n=1}^\infty$, $(h_n)_{n=1}^\infty$ in $X$ such that
$\norm{g_n}_1 \to 1$ and $\norm{h_n}_1 \to 1$, $\frac{1}{2}(g_n+h_n)=f$ for all $n\in\NN$, and $(g_n)_{n=1}^\infty$ does not converge weakly to $f$.
Assume that $\set{g_n:n\in\NN}$ is relatively weakly compact. 
Then by the Eberlein-\v Smulian theorem every subsequence of $(g_n)_{n=1}^\infty$ admits a further subsequence, say $(g_{n_k})$, which weakly converges to some $g\in B_X$. 
It follows that $h_{n_k}=2f-g_{n_k}$ also weakly converges to $2f-g\in B_X$. 
Since $f$ is extreme we get that $g=f$. 
This implies that $(g_n)_{n=1}^\infty$ converges weakly to $f$ which is a contradiction.
Thus $\set{g_n:n\in\NN}$ is not relatively weakly compact.

First, let us finish the proof for $\mu$ being a probability measure.
We may apply the Dunford-Pettis theorem to conclude that $\set{g_n:n\in\NN}$ is not equi-integrable.
Hence, there exists $\varepsilon>0$ such that for every $n$ there are $A_n \subset \Omega$ and $k_n>k_{n-1}$ such that $\mu(A_n)\leq \frac1n$ and $\int_{A_n}\abs{g_{k_n}}\,d\mu\geq \varepsilon$.
When $n$ is large enough that $\norm{g_{k_n}}_1$ and $\norm{h_{k_n}}_1$ are smaller than $1+\frac{\varepsilon}{4}$, it follows that $\int_{\Omega \setminus A_n}\abs{g_{k_n}}\,d\mu\leq 1-\frac{3\varepsilon}{4}$ and
$$
1-\frac{\varepsilon}4\geq \int_{\Omega \setminus A_n} \abs{\frac{g_{k_n}+h_{k_n}}2}\,d\mu=\int_{\Omega \setminus A_n} \abs{f} d\mu \to 1 .
$$ 
This contradiction finishes the proof.

Now assume that $(\Omega,\Sigma,\mu)$ is a general 
measure space.
It follows from the Bohnenblust-Kakutani-Nakano theorem (see \cite[4.8.3.3]{Pietsch_2007} or \cite[p. 136]{Lacey_1974}) that there exist finite measures $(\mu_i)_{i \in I}$ such that
$$
L_1(\mu)=\left(\sum_{i \in I} L_1(\mu_i)\right)_1.
$$
Thus every element $f \in L_1(\mu)$ can be considered as a (possibly transfinite) countably supported sequence of functions which are integrable with respect to different measures $\mu_i$.
We will refer to its support in this sense.

Let $I_0$ be the union of the supports of all involved functions: $f$, $f_n$, $g_n$, $h_n$.
Thus they all belong to $\left(\sum_{i \in I_0} L_1(\mu_i)\right)_1$ and it is clear that this space is isometric to $L_1(\nu)$ for some $\sigma$-finite measure $\nu$.
It is standard (see~\cite[Section 5.1]{AlKa}) that this implies that it is also isometric to $L_1(\kappa)$ for some probability measure $\kappa$.
Therefore, since $f$ is extreme in $\ball{L_1(\nu)}$, the proof of the probability measure case implies that it is preserved extreme.
This contradiction finishes the proof.
\end{proof}

If we consider in particular $X=\lipfree{M}$ when $M$ is a subset of an $\RR$-tree then we get the following consequence:

\begin{corollary}
\label{cr:extreme_points_trees}
Let $M$ be a complete subset of an $\RR$-tree. Then $\gamma \in \F M$ is an extreme point of $B_{\F M}$ if and only if
\begin{equation}
\label{eq:molecule}
\gamma = \frac{\delta(x) - \delta(y)}{d(x,y)}
\end{equation}
for some $x \neq y \in M$ such that $[x,y]\cap M=\set{x,y}$. 
Moreover, all extreme points of $\ball{\lipfree{M}}$ are preserved and exposed.
\end{corollary}

\begin{proof}
The implication ``$\Leftarrow$'' and the fact that any extreme point of the form \eqref{eq:molecule} is exposed have been proved in \cite[Theorem 3.2]{APPP_2019}. 
The implication ``$\Rightarrow$'' follows from Theorem \ref{t:MLURallMeasures}, the fact that $\Free(M) \subseteq L_1(\mu)$ for some measure $\mu$ \cite{Godard_2010} and the fact that, in Lipschitz-free spaces, preserved extreme points of the unit ball have the form \eqref{eq:molecule} \cite[Corollary 3.44]{Weaver2}.
\end{proof}

Note that, although all extreme points of $\ball{\lipfree{M}}$ are preserved and exposed, not all of them are necessarily strongly exposed even if $M$ is compact, as showcased e.g. by \cite[Example 6.4]{GPPR_2018}.

\section*{Acknowledgments}

This research was carried out during a visit of the first author to the Laboratoire de Math\'ematiques de Besan\c{c}on in 2019. He is grateful for the opportunity and the hospitality.

This work was supported by the French ``Investissements d'Avenir'' program, project ISITE-BFC (contract ANR-15-IDEX-03). R. J. Aliaga was also partially supported by the Spanish Ministry of Economy, Industry and Competitiveness under Grant MTM2017-83262-C2-2-P.
The authors would like to thank Abraham Rueda Zoca for his valuable suggestions.


\end{document}